\documentclass[a4paper,20pt,eqno]{article}

\usepackage[latin1]{inputenc}
\usepackage[english]{babel}
\usepackage[]{amsthm}
\usepackage[]{amsmath}
\usepackage[T1]{fontenc}
\usepackage[]{amssymb}
\usepackage[]{mathtools}
\usepackage[]{enumitem}
\usepackage[]{xcolor}

\theoremstyle{remark}

\theoremstyle{plain}

\newtheorem{lem}{Lemma}
\newtheorem{thm}{Theorem}
\newtheorem{prop}{Proposition}

\theoremstyle{definition}
\newtheorem{defin}{Definition}
\newtheorem{hyp}{Assumption}

\newcommand{\F}{\mathcal F}

\newcommand{\Oo}{\mathcal O}
\newcommand{\N}{\mathbb N}
\newcommand{\R}{\mathbb R}

\newcommand{\E}{\mathcal E}

\title{Maximal $L^2$ regularity for Ornstein-Uhlenbeck equation in convex sets of Banach spaces}
\author{G. Cappa\footnote{Dipartimento di Matematica e Informatica, Università degli Studi di Parma, Parco Area delle Scienze 53/A, PARMA, Italy.
E-mail address: gianluca.cappa@nemo.unipr.it}}

\begin{document}
\maketitle



\begin{abstract}
We study the elliptic equation $\lambda u-L^{\Omega}u=f$ in an open convex subset $\Omega$ of an infinite dimensional separable Banach space $X$ endowed with a centered non-degenerate Gaussian measure $\gamma$, where $L^\Omega$ is the Ornstein-Uhlenbeck operator. We prove that for $\lambda>0$ and $f\in L^2(\Omega,\gamma)$ the weak solution $u$ belongs to the Sobolev space $W^{2,2}(\Omega,\gamma)$. Moreover we prove that $u$ satisfies the Neumann boundary condition in the sense of traces at the boundary of $\Omega$. This is done by finite dimensional approximation.
\end{abstract}

\section{Introduction}

Let $X$ be a separable Banach space, let $\gamma$ be a centered non-degenerate Gaussian measure in $X$ with covariance $Q$, and let $H=Q^{1/2}(X)$ be the associated Cameron-Martin space. In this paper we consider the equation
\begin{equation}
 \lambda u-L^\Omega u=f \quad\text{in }\Omega,
 \label{Neum_probl_infinite_dim}
\end{equation}
where $\lambda>0$ and $f\in L^2(\Omega,\gamma)$ are given, $\Omega$ is an open convex set of $X$ and $L^\Omega$ is the Ornstein-Uhlenbeck operator associated to the quadratic form
\[\E_{\Omega,\gamma}(u,v):=\int_\Omega\langle \nabla_Hu,\nabla_Hv\rangle_H d\gamma\quad\text{for}\quad u,v\in W^{1,2}(\Omega,\gamma).\]
Precise definition of the Sobolev spaces $W^{1,2}(\Omega,\gamma)$, $W^{2,2}(\Omega,\gamma)$, and of the $H-$gradient $\nabla_H$ are in the next section.
As usual a function $u\in W^{1,2}(\Omega,\gamma)$ is called  weak solution to \eqref{Neum_probl_infinite_dim} if
\[\int_\Omega\left(\lambda u\varphi+\langle \nabla_Hu,\nabla_H\varphi\rangle_H\right)d\gamma=\int_\Omega f\varphi\ d\gamma\quad\forall\varphi\in W^{1,2}(\Omega,\gamma).\]
It is not hard to see that for every $\lambda>0$ and $f\in L^2(\Omega,\gamma)$, problem \eqref{Neum_probl_infinite_dim} has a unique weak solution $u$.
In this paper we prove a maximal regularity result for the weak solution $u$ of \eqref{Neum_probl_infinite_dim}, that is for every $f\in L^2(\Omega,\gamma)$ the weak solution $u$ belongs to $W^{2,2}(\Omega,\gamma)$ and there exists $C>0$ independent of $f$ such that
\begin{equation}
  \|u\|_{W^{2,2}(\Omega,\gamma)}\leq C\|f\|_{L^2(\Omega,\gamma)}.
  \label{maximal_L2_regularity}
\end{equation}

It is sufficient to have that \eqref{maximal_L2_regularity} holds if $F$ is a cylindrical smooth bounded function (see \textbf{Section \ref{preliminari}}), because the space of such functions is dense in $L^2(\Omega,\gamma)$. In this case, we define a sequence of functions $\{u_n\}_{n\in\N}$, by using the cylindrical approximation $\{\Omega_n\}_{n\in\N}$ of $\Omega$ made in \cite{LunMirPal}. In particular,
\[u_n=\varphi_n\circ\pi_n\]
where $\pi_n(X)$ is a finite dimensional subspace of $H$, identified in an obvious way with $\R^{q}$ with $q=q(n,f)$. So $\pi_n(\Omega_n)$ is identified with an open subset $\Oo_n$ of $\R^{q}$, and $\varphi_n:\Oo_n\subset\R^{q}\rightarrow\R$ solves
\begin{equation}
\left\{
 \begin{split}
  &\lambda \psi-L^{\Oo_n} \psi=\widetilde{f} &\text{in }&\Oo_n\subset \R^{q},\\
  &\frac{\partial \psi}{\partial\nu}=0 &\text{on }&\partial\Oo_n
 \end{split}\right.
 \label{Neum_probl_finite_dim_prototipo}
\end{equation}
where $\widetilde{f}$ is a suitable smooth bounded function. Here, the reference measure is the standard Gaussian measure $N_{0,I}$, and $\nabla_H$ is the usual gradient. For the finite dimensional problems \eqref{Neum_probl_finite_dim_prototipo} we prove dimension free $W^{2,2}$ estimates. Therefore the sequence  $\{u_n\}_{n\in\N}$ is bounded in $W^{2,2}(\Omega,\gamma)$, and a subsequence weakly converges to $u\in W^{2,2}(\Omega,\gamma)$. Eventually we prove that $u$ is a weak solution of \eqref{Neum_probl_infinite_dim}.

Moreover, under some regularity assumption on the boundary of $\Omega$, we prove that the weak solution of \eqref{Neum_probl_infinite_dim} satisfies
\begin{equation}
  \label{cond_neumann}
  \langle\nabla_Hu,\nabla_Hg\rangle_H=0
\end{equation}
on $\partial\Omega$, in the sense of traces. This identity plays the role of the Neumann boundary condition. We use the same sequence $\{u_n\}_{n\in\N}$ defined above, and we show that
\[\int_{\Omega}(\lambda u_n-L^{\Omega_n}u_n)\varphi\ d\gamma=\int_\Omega f\varphi\ d\gamma,\]
for all smooth cylindrical functions $\varphi$, where $L^{\Omega_n}$ is the Ornstein-Uhlenbeck operator associated to the quadratic form $\E_{\Omega_n,\gamma}$. Applying the integration by parts formula \eqref{eq:formuala_itegr_part} we get
\[\int_\Omega\lambda\varphi u_n\ d\gamma+\int_\Omega\langle\nabla_H u_n,\nabla_H \varphi\rangle_H d\gamma=\int_\Omega f\varphi\ d\gamma+\int_{\partial\Omega}\langle\nabla_H u_n,\frac{\nabla_H g}{|\nabla_H g|_H}\rangle_H \varphi\ d\rho,\]
here $g:X\rightarrow\R$ is a suitable convex function such that $g^{-1}(0)=\partial\Omega$ and $\rho$ is the surface measure associated to the Gaussian measure, see \cite{feyel}. Taking the limit along a weakly convergent subsequence, we obtain
\[\int_\Omega\lambda\varphi u\ d\gamma+\int_\Omega\langle\nabla_H u,\nabla_H \varphi\rangle_H d\gamma=\int_\Omega f\varphi\ d\gamma+\int_{\partial\Omega}\langle\nabla_H u,\frac{\nabla_H g}{|\nabla_H g|_H}\rangle_H \varphi\ d\rho,\]
for all smooth cylindrical functions $\varphi$. Since $u$ is the weak solution of \eqref{Neum_probl_infinite_dim} then we can conclude that
\[\int_{\partial\Omega}\langle\nabla_H u,\frac{\nabla_H g}{|\nabla_H g|_H}\rangle_H \varphi\ d\rho=0\]
for all smooth cylindrical functions $\varphi$, that is equivalent to \eqref{cond_neumann}.

The maximal $L^p$ regularity for Ornstein-Uhlenbeck equations was established in \cite{Meyer} by Meyer when $\Omega$ is the whole space $X$ for $1<p<\infty$.
When $\Omega\subsetneqq X$ and $p=2$ the maximal regularity problem was also studied in Hilbert spaces by Da Prato and Lunardi in \cite{LunPrat} with Dirichlet boundary condition and in \cite{PratoL_domini_infini} with Neumann boundary condition for a different class of differential operators that doesn't contain the classical Ornstein-Uhlenbeck operator. Also, the proof in \cite{PratoL_domini_infini} is different from ours because it uses a penalization method approaching the weak solution by a sequence of solutions of problems on whole $X$.

In finite dimension more results are available. Maximal $L^p$ regularity, for $p\in(1,\infty)$, was studied by Metafune, Pruess, Rhandi, and Schnaubelt in  \cite{metafune} when $\Omega=\R^n$ for a class of second order differential operators with unbounded coefficients that contains symmetric Ornstein-Uhlenbeck operators. Maximal $L^2$ regularity in open convex sets of $\R^n$, with Neumann boundary condition, was established in \cite{PratoLun} by methods different from ours.

\section{Preliminaries and definitions}
\label{preliminari}
In this section we recall some basic definitions and notations. Hereafter $h_i$ will denote the $i-$th element of an orthonormal basis of $H$; for every $k\in\N$ set $\widehat{h}_k=Q^{-1}(h_k)$ (see \cite[p.~39-40]{bogachev}).
If $x_i\in\R^n$ we denote by $D_i$ the directional derivative in the direction of $x_i$ while by $\partial_i$ we denote the directional derivative in the direction of $h_i$.

\begin{defin}
 $\F C_b^k(X)$ is the space of cylindrical functions of the form
 \[f(x)=\varphi(l_1(x),\ldots,l_n(x)),\]
 with $\varphi\in C^k_b(\R^n)$, $l_i$, $\ldots$, $l_n\in X^*$ and $n\in\N$.
\end{defin}

\begin{defin}
 $W^{1,2}(\Omega,\gamma)$ is the Sobolev space defined as the completion of the restriction to $\Omega$ of the elements of space $\F C^1_b(X)$ with respect to the norm
 \[\|f\|_{W^{1,2}(\Omega,\gamma)}^2=\int_{\Omega}\left(f^2+|\nabla_H f|^2_H\right) d\gamma.\]
 where $\nabla_H$ is the gradient along the direction of $H$.
\end{defin}

\begin{defin}
 $W^{2,2}(\Omega,\gamma)$ is the Sobolev space defined as the completion of the restriction to $\Omega$ of the elements of space $\F C^2_b(X)$ with respect to the norm
 \[\|f\|_{W^{2,2}(\Omega,\gamma)}^2=\int_{\Omega}\left(f^2+|\nabla_H f|^2_H+\|D^2_H f\|^2_\mathcal{H}\right) d\gamma.\]
 where $D^2_H$ is the $H$-Hessian operator and $\|\cdot\|_\mathcal{H}$ is the Hilbert-Schmidt norm.
\end{defin}

\begin{defin}[Weak solution]
  The function $u\in W^{1,2}(\Omega,\gamma)$ is a weak solution of \eqref{Neum_probl_infinite_dim} if
  \begin{equation}
    \label{eq:weak_solution}
    \int_{\Omega}\lambda u\varphi\ d\gamma+\int_{\Omega}\langle\nabla_H u,\nabla_H\varphi\rangle_H d\gamma=\int_{\Omega}f\varphi\ d\gamma\quad\forall\varphi\in\F C^1_b(X)
  \end{equation}
  or equivalently for all $\varphi\in W^{1,2}(\Omega,\gamma)$.
\end{defin}

\begin{hyp}
\label{Ipostesi_su_Omega}
We suppose that $\Omega=g^{-1}(-\infty,0)$, where $g:X\rightarrow\R$ is a continuous function such that
 \begin{itemize}
  \item $g\in W^{2,p}(X,\gamma)$ for all $p>1$;
  \item there exists $\delta>0$ such that $\displaystyle \frac{1}{|\nabla_H g|_H}\in L^p(g^{-1}(-\delta,\delta),\gamma)$ for all $p>1$.
 \end{itemize}
\end{hyp}
These conditions allow the definition of traces of Sobolev functions  at $g^{-1}(0)=\partial\Omega$, see \cite{CeladaLunardi}.

Let $\varphi,\psi\in W^{1,2}(\Omega,\gamma)$, we recall the integration by parts formula:
\begin{equation}
  \label{eq:formuala_itegr_part}
  \int_{\Omega}\partial_k\varphi\ \psi\ d\gamma=-\int_{\Omega}\varphi\partial_k\psi\ d\gamma+\int_{\Omega}\varphi\psi\ \widehat{h}_k d\gamma+\int_{\partial\Omega}\mathrm{Tr}\varphi \mathrm{Tr}\psi\frac{\partial_k g}{|\nabla_Hg|_H}\ d\rho.
\end{equation}
where in the last integral $\rho$ is the surface measure associated to the Gaussian measure and $\mathrm{Tr}\varphi$, $\mathrm{Tr}\psi$ are the traces of the function $\varphi,\psi$ (see \cite{CeladaLunardi}).

In \cite{Cappa} the \emph{Logarithmic-Sobolev inequality} is proved:
 \begin{equation}
   \label{eq:log-sob_infinite_dim}
   \int_\Omega f^2\log(f^2)d\gamma\leq\int_{\Omega}|\nabla_H f|_H^2d\gamma+\|f\|_{L^2(\Omega,\gamma)}^2\log(\|f\|_{L^2(\Omega,\gamma)}^2),
 \end{equation}
that holds for every $f\in W^{1,2}(\Omega,\gamma)$.

For $u,v\in W^{1,2}(\Omega,\gamma)$ let
\[\E_{\Omega,\gamma}(u,v):=\int_\Omega \langle\nabla_Hu,\nabla_Hv\rangle_H\ d\gamma\]
be the quadratic form associated to $\nabla_H$; we set
\begin{equation}
  \begin{split}
  D(L^\Omega)=\bigg\{u\in W^{1,2}(\Omega,\gamma)&:\exists f\in L^2(\Omega,\gamma)\text{ s.t. }.\\
   &\E_{\Omega,\gamma}(u,v)=-\int_\Omega fv\ d\gamma,\ \forall v\in W^{1,2}(\Omega,\gamma)\bigg\}
  \label{def:dominio_L}
\end{split}
\end{equation}
and we put $L^{\Omega}u:=f$.

Let $\Oo$ be a smooth convex set of $\R^n$ and let $\mu$ be the standard Gaussian measure in $\R^n$. Let $L^\Oo$ be the Ornstein-Uhlenbeck operator associated to the quadratic form  $\E_{\Oo,\mu}$. It is known, see \cite{PratoLun}, that
\begin{equation}
  D(L^{\Oo})=\left\{f\in W^{2,2}(\Oo,\mu):\ \langle x,\nabla f\rangle\in L^2(\Oo,\mu)\text{ and }\frac{\partial f}{\partial\nu}=0\right\}
  \label{caratterizzazione_domini_L_dim_finita}
\end{equation}
where $\nu(x)$ is the exterior normal vector to $\partial\Oo$ at $x$. Moreover
\begin{equation}
L^{\Oo}f(x)=\Delta f(x)-\langle x,\nabla f(x)\rangle\text{ for every }f\in D(L^\Oo).
\label{formula_di_L}
\end{equation}

We recall the finite dimensional logarithmic Sobolev inequality
 \begin{equation}
   \label{eq:log-sob_finite_dim}
   \int_\Oo f^2\log(f^2)d\mu\leq\int_{\Oo}|\nabla_H f|_H^2d\mu+\|f\|_{L^2(\Oo,\mu)}^2\log(\|f\|_{L^2(\Oo,\mu)}^2),
 \end{equation}
that holds for each $f\in W^{1,2}(\Oo,\mu)$, see \cite{PratoLun}.

\section{Finite-dimensional estimates}
Let $\Oo$ be an open smooth convex subset of $\R^n$, with fixed $n$. We assume that
\[\Oo=\{x\in\R^n:\ g(x)<0\}\]
where $g$ is a smooth convex function whose gradient does not vanish at the boundary $\partial\Oo$. We denote by $\nu(x)$ the exterior normal vector to $\partial\Oo$ at $x$, $\nu(x)=\frac{\nabla g(x)}{|\nabla g(x)|}$.
Let $\mu$ be the standard Gaussian measure in $\R^n$ and let $L^\Oo$ be the associated Ornstein-Uhlenbeck operator, that is
\[L^\Oo\psi(x)=\sum_{i=1}^n D_{ii}\psi(x)-\sum_{i=1}^n x_i D_i\psi(x)\text{ for }\psi\in D(L^\Oo).\]

In this section we consider the following problem
\begin{equation}
  \label{Neum_probl_finite_dim}
\left\{\begin{split}
  &\lambda \psi-L^{\Oo} \psi=f &\text{in }&\Oo\subset R^n,\\
  &\frac{\partial u}{\partial\nu}=0 &\text{on }&\partial\Oo
 \end{split}\right.
\end{equation}
where $f\in L^2(\Oo,\mu)$ and $\lambda>0$.

Let us introduce a weighted surface measure on $\partial\Oo$:
\[d\sigma=N(x)dH^{n-1}\]
where $N(x)=(2\pi)^{-n/2}\exp(-|x|^2/2)$ is the Gaussian weight and $dH^{n-1}$ is the usual Hausdorff $(n-1)$ dimensional measure. Using the surface measure $d\sigma$ the integration by parts formula reads as:
\begin{equation}
\int_{\Oo}D_k\varphi\psi\ d\mu=-\int_{\Oo}\varphi D_k\psi\ d\mu+\int_{\Oo}x_k\varphi\psi\ d\mu+\int_{\partial\Oo} \frac{D_kg}{|\nabla g|}\varphi\psi\ d\sigma,
\label{integr_by_parts_formula}
\end{equation}
for each $\varphi,\psi\in W^{1,2}(\Oo,\mu)$ one of which with bounded support, so the boundary integral is meaningful. Indeed $W^{1,2}(\Oo,\mu)\subset W^{1,2}_{loc}(\Oo,dx)$ and the trace at the boundary of any function in $W^{1,2}(\Oo,\mu)$ belongs to $L^2_{loc}(\partial\Oo,dH^{n-1})=L^2_{loc}(\partial\Oo,d\sigma)$.

Applying \eqref{integr_by_parts_formula} with $\varphi$ replaced by $D_k\varphi$ and summing up, we find
\begin{equation}
\int_{\Oo}L^\Oo\varphi\psi\ d\mu=-\int_{\Oo}\langle \nabla\varphi,\nabla\psi\rangle d\mu+\int_{\partial\Oo}\frac{\langle \nabla\varphi,\nabla g\rangle}{|\nabla g|}\psi\ d\sigma
\label{integr_formula}
\end{equation}
for every $\varphi\in W^{2,2}(\Oo,\mu)$ , $\psi\in W^{1,2}(\Oo,\mu)$ one of which with bounded support.

Now we give dimension free estimates for the weak solution $u\in W^{1,2}(\Oo,\mu)$ to \eqref{Neum_probl_finite_dim} with $\lambda>0$ and $f\in L^2(\Oo,\mu)$. We can suppose $f\in C^\infty_c(\Oo)$ because $C^\infty_c(\Oo)$ is dense in $L^2(\Oo,\mu)$. In this case, thanks to classical results about elliptic equations with smooth coefficients we know that the weak solution $u$ of \eqref{Neum_probl_finite_dim} belongs to $C^\infty(\overline{\Oo})\subset W^{2,2}_{loc}(\Oo,\mu)$. Since $\Oo$ can be unbounded, we introduce a smooth cutoff function $\theta:\R^n\rightarrow\R$ such that
\[0\leq\theta(x)\leq1,\ |\nabla\theta(x)|\leq2\ \forall x\in\R^n,\quad \theta\equiv1\text{ in }B(0,1),\quad\theta\equiv0\text{ outside }B(0,2)\]
and we set, for $R>0$
\[\theta_R(x)=\theta(x/R),\quad x\in\R^n.\]

For the $W^{1,2}$ estimates we take $u$ as a test function in the definition of weak solution and we get
\begin{equation}
  \lambda\int_{\Oo}u^2d\mu+\int_{\Oo}|\nabla u|^2d\mu=\int_{\Oo}fu\ d\mu,
  \label{eq:u_soluzione_debole}
\end{equation}
then
\begin{equation}
 \int_{\Oo}u^2d\mu\leq\frac{1}{\lambda^2}\|f\|^2_{L^2(\Oo,\mu)},\quad\int_{\Oo}|\nabla u|^2d\mu\leq\frac{1}{\lambda}\|f\|^2_{L^2(\Oo,\mu)}.
 \label{stime_regolari}
\end{equation}

The following lemma takes into the account the convexity of $\Oo$.

\begin{lem}
\label{lemma:bordo_Omega}
If $u\in C^2(\overline{\Oo})$ satisfies  $\langle \nabla u,\nu\rangle=0$ on $\partial\Oo$ then
\[\langle D^2u\cdot\nabla u,\nu\rangle\leq0\text{ on }\partial\Oo.\]
\end{lem}

\begin{proof}
We recall that $\partial\Oo=g^{-1}(0)$ where $g:\R^n\rightarrow\R$ is a smooth convex function. Let $\tau\in\R^n$ such that $\langle\tau,\nu(x)\rangle=0$ for $x\in\partial\Oo$, then we have
\begin{equation}
  \label{derivata_direzionale_nu}
  \langle\frac{\partial\nu}{\partial\tau}(x),\tau\rangle\geq0\quad\forall x\in\partial\Oo.
\end{equation}
Indeed
\[
\begin{split}
  \frac{\partial\nu}{\partial\tau}&=\frac{\partial}{\partial\tau}\left(\frac{\nabla g}{|\nabla g|}\right)=\frac{1}{|\nabla g|}\frac{\partial}{\partial\tau}(\nabla g)+\frac{\partial}{\partial\tau}\left(\frac{1}{|\nabla g|}\right)\nabla g\\
  &=\frac{1}{|\nabla g|} D^2g\cdot \tau+\langle\nabla\left(\frac{1}{|\nabla g|}\right),\tau\rangle\nabla g,
\end{split}
\]
then for $x\in\partial\Oo$ we have
\[
\begin{split}
   \langle\frac{\partial\nu}{\partial\tau}(x),\tau\rangle&=\frac{1}{|\nabla g(x)|} \langle D^2g(x)\cdot \tau,\tau\rangle +\langle\nabla\left(\frac{1}{|\nabla g(x)|}\right),\tau\rangle \langle\nabla g(x),\tau\rangle\\
   &=\frac{1}{|\nabla g(x)|} \langle D^2g(x)\cdot \tau,\tau\rangle\geq0
\end{split}
\]
since $D^2g$ is a positive semi-definite symmetric matrix. Now we recall that $\langle \nabla u,\nu\rangle=0$ on $\partial\Oo$ therefore
\[
  \frac{\partial}{\partial\tau}(\langle \nabla u(x),\nu(x)\rangle)=\langle D^2u(x)\ \tau,\nu(x)\rangle+\langle\nabla u(x),\frac{\partial\nu}{\partial\tau} (x)\rangle=0,\quad x\in\partial\Oo
\]
for each $\tau\in\R^n$ such that $\langle\tau,\nu\rangle=0$ on $\partial\Oo$. If we take $\tau=\nabla u(x)$ then we get
\[\langle D^2u(x)\cdot\nabla u(x),\nu(x)\rangle=-\langle\tau,\frac{\partial\nu}{\partial\tau} (x)\rangle\leq0,\quad x\in\partial\Oo.\]
\end{proof}

Now we can give an estimate of the second order derivatives of $u$.

\begin{prop}
  For every $f\in C_c^{\infty}(\Oo)$ and $\varepsilon>0$ there exists $R_0>0$ such that for $R>R_0$ the solution $u$ to \eqref{Neum_probl_finite_dim} satisfies
  \begin{equation}
    \left(1-\varepsilon\right)\int_{\Oo}\theta_R^2 \mathrm{Tr}[(D^2u)^2]d\mu\leq\left(2+\frac{\varepsilon}{\lambda}\right)\|f\|^2_{L^2(\Oo,\mu)}.
    \label{eq:stima_derivate_seconde_epsilon_dim_finita}
  \end{equation}
\end{prop}

\begin{proof}
  Recall that $u\in C^\infty(\overline{\Oo})$; differentiating \eqref{Neum_probl_finite_dim}
with respect to $x_h$ yields
\[\lambda D_hu-\Delta D_hu-\langle x,\nabla(D_hu)\rangle+D_hu=D_hf.\]
Multiplying by $D_hu\theta_R^2$ we obtain
\[\left(\lambda+1\right)(D_hu)^2\theta_R^2- \Delta D_hu\cdot D_hu\theta_R^2-\langle x,\nabla(D_hu)\rangle D_hu\theta_R^2=D_hfD_hu\theta_R^2.\]
Integrating over $\Oo$ and using \eqref{integr_formula} yields
\[
 \begin{split}
  &\int_{\Oo}(\lambda+1)(D_hu)^2\theta^2_Rd\mu+\int_{\Oo} |\nabla(D_hu)|^2 \theta_R^2d\mu+2\int_{\Oo}\theta_R\langle  \nabla(D_hu),\nabla\theta_R\rangle D_hu\ d\mu\\
  &=\int_{\partial\Oo}\frac{\langle  \nabla(D_hu),\nabla g\rangle D_hu}{|\nabla g|}\theta_R^2d\sigma+\int_{\Oo} D_hfD_hu\theta_R^2d\mu.
 \end{split}
\]
Summing over $h$ we obtain
\[
 \begin{split}
  &\int_{\Oo}(\lambda+1)|\nabla u|^2\theta^2_Rd\mu+\int_{\Oo}\mathrm{Tr}[(D^2u)^2]\theta_R^2d\mu+2\int_{\Oo}\langle  D^2u\cdot \nabla u,\nabla\theta_R\rangle \theta_R\ d\mu\\
  &=\int_{\partial\Oo}\frac{\langle  D^2u\cdot \nabla u,\nabla g\rangle}{|\nabla g|}\theta_R^2d\sigma+\int_{\Oo}\langle  \nabla f,\nabla u\rangle\theta_R^2d\mu.
 \end{split}
\]
Since $f$ has compact support, for $R$ large enough $\theta_R\equiv1$ on the support of $f$. For such $R$ we obtain
\[\left|\int_{\Oo}\langle  \nabla f,\nabla u\rangle\theta_R^2d\mu\right|=\left|-\int_{\Oo}L^\Oo u fd\mu\right|=\left|\int_{\Oo}(\lambda u-f) fd\mu\right|\leq 2\|f\|_{L^2(\Oo,\mu)}^2.\]
Moreover
\[
\begin{split}
    &\left|\int_{\Oo}\langle  D^2u \nabla u,\nabla\theta_R\rangle \theta_R\ d\mu\right| \leq\int_\Oo \sum_{i,j=1}^n |D_{ij}u\ D_ju\ D_i\theta_R|\theta_R\ d\mu\\
    &\leq\frac{1}{2}\int_\Oo \sum_{i,j=1}^n |D_{ij}u|^2 |D_i\theta_R|\theta_R^2\ d\mu+\frac{1}{2}\int_\Oo \sum_{i,j=1}^n |D_ju|^2|D_i\theta_R|\ d\mu\\
    &\leq\frac{1}{2}\frac{\||\nabla\theta|\|_\infty}{R} \int_\Oo \theta_R^2 \mathrm{Tr}[(D^2u)^2]d\mu+\frac{1}{2}\frac{\||\nabla\theta|\|_\infty}{R}\int_\Oo|\nabla u|^2 d\mu\\
    &\leq\left(\frac{1}{2}\int_{\Oo}\mathrm{Tr}[(D^2u)^2]\theta_R^2d\mu+\frac{1}{2\lambda}\|f\|_{L^2(\Oo,\mu)}^2\right)\frac{\||\nabla\theta|\|_\infty}{R}.
\end{split}
\]
Taking $R$ large enough, such that $\||\nabla\theta|\|_\infty/R\leq\varepsilon$, we get
\[\left(1-\varepsilon\right)\int_{\Oo}\theta_R^2 \mathrm{Tr}[(D^2u)^2]d\mu\leq\left(2+\frac{\varepsilon}{\lambda}\right)\|f\|^2_{L^2(\Oo,\mu)} +\int_{\partial\Oo}\theta^2_R\frac{\langle D^2u\cdot \nabla u,\nabla g\rangle}{|\nabla g|}d\sigma.\]
By using \textbf{Lemma \ref{lemma:bordo_Omega}}, the statement follows.
\end{proof}

\begin{thm}
  \label{thm:stima_max_regolarita}
  For each $\lambda>0$ there exists $C=C(\lambda)>0$, independent of $n$ and $\Oo$, such that for each $f\in L^2(\Oo,\mu)$ the weak solution $u$ of problem \eqref{Neum_probl_finite_dim} belongs to $W^{2,2}(\Oo,\mu)$, and satisfies
  \begin{equation}
    \|u\|_{W^{2,2}(\Oo,\mu)}\leq C\|f\|_{L^2(\Oo,\mu)}.
    \label{eq:stima_max_reg_dim_finita}
  \end{equation}
\end{thm}

\begin{proof}
  Let $f\in C^\infty_c(\Oo)$. Taking the limit as $R\rightarrow\infty$ in \eqref{eq:stima_derivate_seconde_epsilon_dim_finita} and using the monotone convergence theorem, we get
  \[\left(1-\varepsilon\right)\int_{\Oo}\mathrm{Tr}[(D^2u)^2]d\mu\leq\left(2+\frac{\varepsilon}{\lambda}\right)\|f\|^2_{L^2(\Oo,\mu)}.\]
  Now taking the limit as $\varepsilon\rightarrow0$ we get
  \begin{equation}
    \int_{\Oo}\mathrm{Tr}[(D^2u)^2]d\mu\leq2\|f\|^2_{L^2(\Oo,\mu)}.
    \label{eq:stima_derivate_seconde_dim_finita}
  \end{equation}
  Taking into account \eqref{eq:u_soluzione_debole}, \eqref{stime_regolari}, and \eqref{eq:stima_derivate_seconde_dim_finita} we obtain
  \[\begin{split}
    \|u\|_{W^{2,2}(\Oo,\mu)}^2&=\|u\|_{L^2(\Oo,\mu)}+\||\nabla u|\|_{L^2(\Oo,\mu)}+\|\mathrm{Tr}[(D^2u)^2]\|_{L^2(\Oo,\mu)}\\
    &\leq\left(\frac{1}{\lambda^2}+\frac{1}{\lambda}+2\right)\|f\|^2_{L^2(\Oo,\mu)}
  \end{split}\]
  which is \eqref{eq:stima_max_reg_dim_finita} with $C(\lambda)=\frac{1}{\lambda^2}+\frac{1}{\lambda}+2$.
  For $f\in L^2(\Oo,\mu)$ the statement follows approaching it by a sequence of functions belonging to $C^\infty_c(\Oo)$.
\end{proof}

Now we get a characterization of the domain of $L^\Oo$. We recall that \eqref{caratterizzazione_domini_L_dim_finita} holds, and we prove that the condition $\langle \cdot,\nabla f\rangle\in L^2(\Oo,\mu)$ can be omitted.

\begin{prop}
  If $f\in W^{2,2}(\Oo,\mu)$ then $\langle x,\nabla f\rangle\in L^2(\Oo,\mu)$, moreover the map
  \[f\mapsto\langle\cdot,\nabla f\rangle\]
  is continuous from $W^{2,2}(\Oo,\mu)$ to $L^2(\Oo,\mu)$.
\end{prop}

\begin{proof}
 Let $f\in W^{2,2}(\Oo,\mu)$, then
 \[\int_\Oo|\langle\nabla f,x \rangle|^2 d\mu=\int_\Oo \sum_{i=1}^n(D_i f x_i)^2 d\mu\]
 by assumption $D_i f\in W^{1,2}(\Oo,\mu)$ and if $c<1/4$, by using \eqref{eq:log-sob_finite_dim}, we have
 \[\begin{split}
   \int_{\Oo}&(D_i f(x))^2 x_i^2e^{-|x|^2/2}dx\\
   &=\int_{\{x\in\Oo:\ cx_i^2>\log|D_if(x)|\}}(D_i f(x))^2 x_i^2e^{-|x|^2/2}dx\\
   &\phantom{=}+\int_{\{x\in\Oo:\ cx_i^2\leq\log|D_if(x)|\}}(D_i f(x))^2 x_i^2e^{-|x|^2/2}dx\\
   &\leq\int_\Oo e^{2cx_i^2}x_i^2e^{-|x|^2/2}dx+\int_\Oo\frac{1}{c}|D_if|^2\log|D_if|e^{-|x|^2/2}dx\\
   &\leq C_1+\frac{1}{c}\left(\int_\Oo |\nabla D_if|d\mu+\frac{1}{2}\int_\Oo (D_if)^2 d\mu\log\left(\int_\Oo (D_if)^2 d\mu\right)\right).
 \end{split}
 \]
 Summing over $i$ from $1$ to $n$ we have $\langle\nabla f,x \rangle\in L^2(\Oo,\mu)$.
\end{proof}

\section{Approximation by cylindrical functions}
\label{sez:approssimazione_cilind}
Now we consider the infinite dimensional problem. Let $\Omega\subset X$ be an open convex set and let $\{\Omega_n\}$ be a sequence of cylindrical open convex sets, defined in \cite{LunMirPal}, of the form $\Omega_n=\pi_n^{-1}(\Oo_n)$ where $\Oo_n\subset F_n$, $F_n$ a is finite dimensional subspace of $Q(X^*)\subset H$ with $\dim F_n=j(n)\leq n$, $F_n\subset F_{n+1}$ for $n\in\N$, and $\pi_n:X\rightarrow F_n$ is the projection defined below. Let $\{h_n\}_{n\in\N}\subset Q(X^*)$ be an orthonormal basis of the Cameron-Martin space $H$ such that $F_n=\mathrm{span}\{h_1,\ldots,h_{j(n)}\}$. Therefore
\[\pi_n(x)=\sum_{i=1}^{j(n)}\widehat{h}_i(x)h_i.\]
Moreover $\Omega_{n+1}\subset\Omega_n$, $\partial\Oo_n$ is smooth, $\Omega\subset\Omega_n$ and
\[\overline{\Omega}=\bigcap_{n\in\N}\overline{\Omega}_n,\quad \gamma\left(\bigcap_{n\in\N}\Omega_n\backslash\Omega\right)=0.\]
We recall that since $\Omega$ and $\Omega_n$ are open convex sets, then $\gamma(\partial\Omega)=\gamma(\partial\Omega_n)=0$.

Now we show that the restriction to $\Omega$ of cylindrical continuous smooth functions is dense in $L^2(\Omega,\gamma)$. Let $\psi\in L^2(\Omega,\gamma)$, then the zero extension outside $\Omega$, $\overline{\psi}$, belongs to $L^2(X,\gamma)$. We have from \cite[\textbf{Corollary 3.5.2}]{bogachev} that there exists a sequence of $L^2$ cylindrical functions $\psi_n$ that converges to $\overline{\psi}$ in $L^2(X,\gamma)$. In its turn , each $\psi_n$ may be approached by a sequence of $C^\infty_b$ functions.

Therefore we suppose $f\in\F C^\infty_b(X)$. Then, for some $k\in\N$,
\[f(x)=w(l_1(x),\ldots,l_k(x))\]
where $w\in C^\infty_b(\R^k)$, $l_i\in X^*$ for $i=1,\ldots,k$.

Let $G=G(n,f):=\mathrm{span}\{F_n,Q(l_1),\ldots,Q(l_k)\}$. Then $G$ is a subspace of $H$ of dimension $q=q(n,f)\leq j(n)+k$; setting $d:=q-j(n)$ let $\Oo=\Oo(n,f):=\Oo_n\times \R^d$. If we denote by
\[\pi_G (x)=\sum_{i=1}^q\widehat{h}_i(x)h_i\]
then
\[f(x)=\widetilde{f}(\pi_G(x))\]
where $\widetilde{f}\in C^\infty_b(G)$. Let $\gamma_G$ be the induced measure $\gamma\circ \pi_G^{-1}$ in $G$; if $G$ is identified with $\R^q$ through  the isomorphism $x\mapsto(\widehat{h}_1(x),\ldots,\widehat{h}_q(x))$ then $\gamma_G$ is the standard Gaussian measure in $\R^q$.

We recall that $L^{\Omega_n}$ is the Ornstein-Uhlenbeck operator associated to the quadratic form $\E_{\Omega_n,\gamma}$ while $L^{\Oo}$ is Ornstein-Uhlenbeck operator associated to the quadratic form $\E_{\Oo,\gamma_G}$.

\begin{prop}
  Let $v$ be the weak solution of the finite dimensional problem
  \[\lambda v-L^{\Oo}v=\widetilde{f}_{|\Oo} \quad\text{in }\Oo\]
  Then $u(x):=v(\pi_G(x))$ is the weak solution of
  \[\lambda u-L^{\Omega_n}u=f_{|\Omega_n} \text{in }\Omega_n\]
\end{prop}

\begin{proof}
 We remark that the space $X$ can be split as $X=G\times\widetilde{X}$ where $\widetilde{X}=(I-\pi_G)(X)$, and $\gamma=\gamma_G\otimes\widetilde{\gamma}$ where $\widetilde{\gamma}=\gamma\circ(I-\pi_G)^{-1}$ is the measure induced on $\widetilde{X}$ by the projection $I-\pi_G$. Let $\varphi\in W^{1,2}(\Omega_n,\gamma)$, then
 \begin{equation*}
   \begin{split}
     \int_{\Omega_n}&\left(\lambda u(x)\varphi(x)+\langle\nabla_H u(x),\nabla_H\varphi(x)\rangle_H\right) \gamma(dx)\\
     &= \int_{\Omega_n}\lambda v(\pi_G(x))\varphi(\pi_G(x)+(I-\pi_G)(x))\\
     &\phantom{=}+\langle\nabla_H v(\pi_G(x)),\nabla_H\varphi(\pi_G(x)+(I-\pi_G)(x))\rangle_H \gamma(dx)\\
     &=\int_{\Oo\times\widetilde{X}}\lambda v(\xi)\widetilde{\varphi}(\xi+y)\\
     &\phantom{=}+\langle\nabla v(\xi), \nabla\widetilde{\varphi}(\xi+y)\rangle\gamma_G(d\xi)\widetilde{\gamma}(dy)\quad(\text{where } \widetilde{\varphi}(\cdot+y)\in W^{1,2}(\Oo,\gamma_G))\\
     &=\int_{\Oo\times\widetilde{X}} \widetilde{f}(\xi)\widetilde{\varphi}(\xi+y)\gamma_G(d\xi)\widetilde{\gamma}(dy)\\
     &=\int_{\Omega_n} \widetilde{f}(\pi_G(x))\varphi(x)\gamma(dx)\\
     &=\int_{\Omega_n} f(x)\varphi(x)\gamma(dx),\\
    \end{split}
  \end{equation*}
  and the statement follows.
\end{proof}

\begin{prop}
  \label{prop:stima_W22_per_u_dim_infinta}
  The function $u$ satisfies
  \[\|u\|_{W^{2,2}(\Omega,\gamma)}\leq K\]
  where
  \[K:=C\|f\|_{L^2(\Omega_1,\gamma)}\]
  and $C$ is the constant of \textbf{Theorem \ref{thm:stima_max_regolarita}}.
\end{prop}

\begin{proof}
  We recall that $u(x)=v(\pi_G(x))$. Then
  \begin{equation*}
    \begin{split}
      \|u\|^2&_{W^{2,2}(\Omega,\gamma)}\leq\|u\|^2_{W^{2,2}(\Omega_n,\gamma)}\\
      &=\int_{\Omega_n}|u(x)|^2+\sum_{i=1}^\infty|D_i u(x)|^2+\sum_{i,j=1}^\infty|D_{ij} u(x)|^2\ \gamma(dx)\\
      &=\int_\Oo\left(|v(\xi)|^2+\sum_{i=1}^{q}|\partial_i v(\xi)|^2\right.\\
      &\phantom{=}\left.+\sum_{i,j=1}^q|\partial_{ij} v(\xi)|^2\right) \mu(d\xi)\text{ (By using \textbf{Theorem \ref{thm:stima_max_regolarita}})}\\
      &\leq C^2\|\widetilde{f}\|_{L^2(\Oo,\mu)}^2=C^2\|f\|^2_{L^2(\Omega_n,\gamma)}\leq C^2\|f\|_{L^2(\Omega_1,\gamma)}^2
    \end{split}
  \end{equation*}
\end{proof}

If we consider the sequence $\{u_n\}_{n\in\N}$ of weak solutions of the problems
\[\lambda \psi-L^{\Omega_n}\psi=f_{|\Omega_n}\quad\text{in }\Omega_n.\]
By \textbf{Proposition \ref{prop:stima_W22_per_u_dim_infinta}} it follows
\[\|u_n\|_{W^{2,2}(\Omega,\gamma)}\leq K.\]
Possibly replacing $u_n$ by a subsequence, there exists $u\in W^{2,2}(\Omega,\gamma)$ such that $u_n\rightharpoonup u$ in $W^{2,2}(\Omega,\gamma)$.

\begin{prop}
  The function $u$ is the weak solution of \eqref{Neum_probl_infinite_dim}.
\end{prop}

\begin{proof}
  We know that for all $\varphi\in \F C^1_b(X)$
  \[\int_{\Omega_n}\lambda u_n\varphi\ d\gamma+\int_{\Omega_n}\langle\nabla_H u_n,\nabla_H\varphi\rangle_H d\gamma=\int_{\Omega_n}f\varphi\ d\gamma.\]
  We claim that
  \[\lim_{n\rightarrow\infty}\int_{\Omega_n}\lambda u_n\varphi\ d\gamma=\int_{\Omega}\lambda u\varphi\ d\gamma.\]
  Indeed,
  \begin{equation}
    \label{eq:per_sol_debole_inf_dim}
    \int_{\Omega_n}\lambda u_n\varphi\ d\gamma=\int_{\Omega}\lambda u_n\varphi\ d\gamma+\int_{\Omega\backslash\Omega_n}\lambda u_n\varphi\ d\gamma;
  \end{equation}
  by the weak convergence
  \[\lim_{n\rightarrow\infty}\int_{\Omega}\lambda u_n\varphi\ d\gamma=\int_{\Omega}\lambda u\varphi\ d\gamma\]
  while
  \[\left|\int_{\Omega\backslash\Omega_n}\lambda u_n\varphi\ d\gamma\right|\leq\lambda\left(\int_{\Omega\backslash\Omega_n} |u_n|^2d\gamma\right)^{1/2}\left(\int_{\Omega\backslash\Omega_n} |\varphi|^2d\gamma\right)^{1/2}\leq\lambda K\left(\int_{\Omega\backslash\Omega_n} |\varphi|^2d\gamma\right)^{1/2}\]
  that goes to zero as $n\rightarrow\infty$ by the absolute continuity of the integral.
  Now we claim that
  \[\lim_{n\rightarrow\infty}\int_{\Omega_n}\langle\nabla_H u_n,\nabla_H\varphi\rangle_H d\gamma=\int_{\Omega}\langle\nabla_H u,\nabla_H\varphi\rangle_H d\gamma.\]
  In fact,
  \[\int_{\Omega_n}\langle\nabla_H u_n,\nabla_H\varphi\rangle_H d\gamma=\int_{\Omega}\langle\nabla_H u_n,\nabla_H\varphi\rangle_H d\gamma+\int_{\Omega\backslash\Omega_n}\langle\nabla_H u_n,\nabla_H\varphi\rangle_H d\gamma.\]
  By the weak convergence in $W^{1,2}(\Omega,\gamma)$
  \[\lim_{n\rightarrow\infty}\int_{\Omega}\langle\nabla_H u_n,\nabla_H\varphi\rangle_H d\gamma=\int_{\Omega}\langle\nabla_H u,\nabla_H\varphi\rangle_H d\gamma\]
  while
  \[
  \begin{split}
    \left|\int_{\Omega\backslash\Omega_n}\langle\nabla_H u_n,\nabla_H\varphi\rangle_H d\gamma\right|&\leq\lambda\left(\int_{\Omega\backslash\Omega_n} |\nabla_Hu_n|_H^2d\gamma\right)^{1/2}\left(\int_{\Omega\backslash\Omega_n} |\nabla_H\varphi|_H^2d\gamma\right)^{1/2}\\
    &\leq\lambda K\left(\int_{\Omega\backslash\Omega_n} |\nabla_H\varphi|_H^2d\gamma\right)^{1/2}
  \end{split}
  \]
  that goes to zero as $n\rightarrow\infty$.

  Moreover,
  \[\lim_{n\rightarrow\infty}\int_{\Omega_n}f\varphi\ d\gamma=\int_{\Omega}f\varphi\ d\gamma.\]
  Therefore letting $n\rightarrow\infty$ in \eqref{eq:per_sol_debole_inf_dim} we get that $u$ satisfies \eqref{eq:weak_solution}.
\end{proof}

Finally we give the maximal regularity estimate.
\begin{thm}
  If $u$ is the weak solution of $\lambda u- L^{\Omega}u=f$ on $\Omega$ then $u\in W^{2,2}(\Omega,\gamma)$ and
  \[\|u\|_{W^{2,2}(\Omega,\gamma)}\leq C\|f\|_{L^2(\Omega,\gamma)}\]
\end{thm}

\begin{proof}
  By \textbf{Proposition \ref{prop:stima_W22_per_u_dim_infinta}} it follows
  \begin{equation}
    \label{eq:stima_u_n_in_Omega_dim_infinita}
    \|u_n\|_{W^{2,2}(\Omega,\gamma)}\leq C\|f\|_{L^2(\Omega_n,\gamma)}
  \end{equation}
  where $C=C(\lambda)$ is the constant of the \textbf{Theorem \ref{thm:stima_max_regolarita}}.

  We remark that
  \[\lim_{n\rightarrow\infty}\|f\|_{L^2(\Omega_n,\gamma)}=\|f\|_{L^2(\Omega,\gamma)}\]
  since $\gamma(\Omega_n\backslash\Omega)\rightarrow0$.

  By the weak convergence of $u_n$ to $u$ we have
  \[\|u\|_{W^{2,2}(\Omega,\gamma)}\leq\limsup_{n\rightarrow\infty}\|u_n\|_{W^{2,2}(\Omega,\gamma)}.\]
  Letting $n\rightarrow\infty$ in \eqref{eq:stima_u_n_in_Omega_dim_infinita} we get our claim.
\end{proof}

\section{The Neumann boundary condition}

In this section we put under \textbf{Assumption \ref{Ipostesi_su_Omega}} and we prove that the weak solution $u$ of \eqref{Neum_probl_infinite_dim} satisfies a Neumann type boundary condition.

First we prove a useful lemma.
\begin{prop}
\label{lem:zero_q_o}
  If $u\in L^p(\partial\Omega,\rho)$ and
  \[\int_{\partial\Omega} u\varphi\ d\rho=0\quad\forall\varphi\in\F C^1_b(X),\]
  then $u=0$ $\rho-$a.e. in $\partial\Omega$.
\end{prop}

\begin{proof}
  Since the map
  \[v\mapsto\int_{\partial\Omega}uv\ d\rho\]
  is continuous from $W^{1,q}(\Omega,\gamma)$ to $\R$ for all $q>p'$, and $\F C^1_b(X)$ is dense in $W^{1,q}(\Omega,\gamma)$, it follows that
  \[\int_{\partial\Omega} u\psi\ d\rho=0\quad\forall\psi\in W^{1,q}(\Omega,\gamma).\]
  In particular, since the restrictions to $\Omega$ of the Lipschitz continuous and bounded functions $\psi:X\rightarrow\R$ belong to $W^{1,q}(\Omega,\gamma)$, we have
  \[\int_{\partial\Omega} u\psi\ d\rho=0\quad\forall\psi\in Lip_b(X).\]
  \textbf{Lemma \ref{lemma:Lip_in_Lp}} yields
  \[\int_{\partial\Omega} u\psi\ d\rho=0\quad\forall\psi\in L^q(\partial\Omega,\rho)\]
  and this implies that $u=0$ $\rho-$a.e..
\end{proof}

Now we are ready to prove that the weak solution of \eqref{Neum_probl_infinite_dim} satisfies a boundary condition similar to the Neumann boundary condition.

\begin{prop}
  If $u$ is the weak solution of $\lambda u- Lu=f$ on $\Omega$ then
  \begin{equation}
  \label{eq:Neumann_bound_condition}
  \langle\nabla_H u(x),\frac{\nabla_H g(x)}{|\nabla_H g(x)|_H}\rangle_H=0\quad\rho-\text{a.e }x\in\partial\Omega.
\end{equation}
\end{prop}

\begin{proof}
  We fix $\varphi\in \F C_b^1(X)$. We denote by $u_n$ the solution to
  \begin{equation}
    \lambda \psi-L^{\Omega_n}\psi=f_{|\Omega_n}\quad\text{in }\Omega_n.
    \label{eq:probl_cilin_dim_infinita}
  \end{equation}
  We recall that $u_n$ is a cylindrical function and, thanks to the result of \textbf{Section \ref{sez:approssimazione_cilind}}, we have $u_n\in W^{2,2}(\Omega_n,\gamma)$. We multiply the differential equation \eqref{eq:probl_cilin_dim_infinita} by $\varphi$ and we integrate on $\Omega$, getting
  \[\int_\Omega(\lambda u_n-L^{\Omega_n}u_n)\varphi\ d\gamma=\int_\Omega f\varphi\ d\gamma.\]
  We recall that $L^{\Omega_n}u_n$ is cylindrical, then
  \[L^{\Omega_n}u_n(x)=\sum_{i=1}^{q}\partial_{ii}u_n(x)-\widehat{h}_i(x)\partial_i u_n(x). \]
  Therefore, by using \eqref{eq:formuala_itegr_part}, we obtain
  \[\int_\Omega\lambda\varphi u_n\ d\gamma+\int_\Omega\langle\nabla_H u_n,\nabla_H \varphi\rangle_H d\gamma=\int_\Omega f\varphi\ d\gamma+\int_{\partial\Omega}\langle\nabla_H u_n,\frac{\nabla_H g}{|\nabla_H g|_H}\rangle_H \varphi\ d\rho,\]
  where
  \[\langle\nabla_H u_n,\nabla_H \varphi\rangle_H=\sum_{i=1}^{q}\partial_i u_n\partial_i\varphi,\]
  and
  \[\langle\nabla_H u_n,\nabla_H g\rangle_H=\sum_{i=1}^{q}\partial_i u_n\partial_i g.\]
  As in the previous section we have
  \[\lim_{n\rightarrow\infty}\int_\Omega\lambda\varphi u_n\ d\gamma=\int_\Omega\lambda\varphi u\ d\gamma,\]
  and
  \[\lim_{n\rightarrow\infty}\int_\Omega\langle\nabla_H u_n,\nabla_H \varphi\rangle_H d\gamma=\int_\Omega\langle\nabla_H u,\nabla_H \varphi\rangle_H d\gamma,\]
  We claim that the map
  \[v\mapsto\int_{\partial\Omega}\langle\nabla_H v,\frac{\nabla_H g}{|\nabla_H g|_H}\rangle_H \varphi\ d\rho\]
  from $W^{2,2}(\Omega,\gamma)$ to $\R$ belongs to $(W^{2,2}(\Omega,\gamma))'$. Indeed, the function
  \[x\mapsto\langle\nabla_H v(x),\frac{\nabla_H g(x)}{|\nabla_H g(x)|_H}\rangle_H \varphi(x)=:F(x)\]
  belongs to $W^{1,q}(\Omega,\gamma)$ for all $q\in(1,2)$. Moreover $\|F\|_{W^{1,q}(\Omega,\gamma)}\leq \widetilde{C}\|v\|_{W^{2,2}(\Omega,\gamma)},$
  and the trace operator is linear and continuous from $W^{1,q}(\Omega,\gamma)$ to $L^1(\partial\Omega,\rho)$.
  Therefore, since $u_n\rightharpoonup u$ in $W^{2,2}(\Omega,\gamma)$,
  \[\lim_{n\rightarrow\infty}\int_{\partial\Omega}\langle\nabla_H u_n,\frac{\nabla_H g}{|\nabla_H g|_H}\rangle_H \varphi\ d\rho=\int_{\partial\Omega}\langle\nabla_H u,\frac{\nabla_H g}{|\nabla_H g|_H}\rangle_H \varphi\ d\rho.\]
  Then we have
  \[\int_\Omega\lambda u\varphi\ d\gamma+\int_\Omega\langle\nabla_H u,\nabla_H \varphi\rangle_H d\gamma=\int_\Omega f\varphi\ d\gamma+\int_{\partial\Omega}\langle\nabla_H u,\frac{\nabla_H g}{|\nabla_H g|_H}\rangle_H \varphi\ d\rho\]
  and since $u$ is a weak solution of \eqref{Neum_probl_infinite_dim} we get
  \[\int_{\partial\Omega}\langle\nabla_H u,\frac{\nabla_H g}{|\nabla_H g|_H}\rangle_H \varphi\ d\rho=0\]
  for all $\varphi\in \F C^1_b(X)$. By using \textbf{Proposition \ref{lem:zero_q_o}} we obtain \eqref{eq:Neumann_bound_condition}.
\end{proof}

Therefore, if $u\in D(L)$ then $u\in W^{2,2}(\Omega,\gamma)$ and $u$ satisfies the Neumann boundary condition \eqref{eq:Neumann_bound_condition}.

\appendix
\section{Density properties}

In this appendix we show some density results for which we thank Simone Ferrari. Let $(Y,d)$ be a complete metric space and let $\rho$ be a finite Radon measure defined on the Borel sets of $Y$. Let $BUC(Y)$ be the set of real value uniformly bounded continuous functions and let $Lip_b(Y)$ be the set of Lipschitz bounded functions.

\begin{lem}
  \label{lemma:appr_BUC}
  Let $f:Y\rightarrow\R$ be a bounded $\rho-$measurable function. Then for all $\varepsilon>0$ there exists $g\in BUC(Y)$ such that
  \[\rho(\{x\in Y:\ f(x)\neq f_\varepsilon(x)\})<\varepsilon\]
  and
  \[\sup_{x\in Y}|g(x)|\leq 2\sup_{x\in Y}|f(x)|.\]
\end{lem}

\begin{proof}
  We fix $\varepsilon>0$. Since $\rho$ is a Radon measure then there exists $K_0$, compact subset of $Y$, such that $\rho(Y\setminus K_0)<\varepsilon$. By the Lusin theorem there exists a function $f_0\in C(K_0)=BUC(K_0)$ such that:
  \[\rho(\{x\in K_0:\ f_0(x)\neq f_{|K_0}(x)\})<\varepsilon\]
  and
  \[\sup_{x\in K_0}|f_0(x)|\leq \sup_{x\in K_0}|f(x)|\leq\sup_{x\in Y}|f(x)|.\]
  We define the following function, studied in \cite{Mandelkern},
  \[g(x)=\left\{
  \begin{split}
    &f(x) &\text{if }x\in K_0\\
    &\inf_{y\in K_0}f_0(y)\frac{d(x,y)}{d(x,K_0)}&\text{if }x\not\in K_0
  \end{split}
  \right.\]
  then $g$ is a $BUC$ extension of $f_0$ to the whole $Y$. We remark that for $x\not\in K_0$ there exists $y_\varepsilon\in K_0$ such that
  \[d(x,K_0)=\inf_{y\in K_0}d(x,y)\geq d(x,y_\varepsilon)-\varepsilon,\]
  therefore for $x\not\in K_0$ we have
  \[\begin{split}
    |g(x)|=\left|\inf_{y\in K_0}f_0(y)\frac{d(x,y)}{d(x,K_0)}\right|\leq \sup_{x\in Y}|f(x)|\frac{d(x,y_\varepsilon)}{d(x,K_0)}\leq \sup_{x\in Y}|f(x)|\frac{d(x,K_0)+\varepsilon}{d(x,K_0)}
  \end{split}\]
  for all $\varepsilon$. Then for all $x\not\in K_0$ we have
  \[g(x)\leq \sup_{y\in Y} |f(y)|.\]
  Finally
  \[\begin{split}
    \sup_{x\in Y}|g(x)|&=\sup_{x\in Y} |g_{|K_0}(x)+g_{|Y\setminus K_0}(x)|\leq\sup_{x\in K_0}|g(x)|+\sup_{x\in Y\setminus K_0}|g(x)|\\
    &=\sup_{x\in K_0}|f_0(x)|+\sup_{x\in Y\setminus K_0}|g(x)|\leq 2\sup_{x\in Y}|f(x)|.
  \end{split}\]
  Moreover
  \[\begin{split}
    \rho(\{x\in Y:\ g(x)\neq f(x)\})&\leq \rho(\{x\in K_0:\ g(x)\neq f(x)\})+\rho(\{x\in Y\setminus K_0:\ g(x)\neq f(x)\})\\
    &\leq \rho(\{x\in K_0:\ f_0(x)\neq f(x)\})+\rho(Y\setminus K_0)<2\varepsilon.
  \end{split}
  \]
\end{proof}

\begin{lem}
\label{lemma:Lip_in_Lp}
  The subspace $Lip_b(Y)$ is dense in $L^p(Y,\rho)$ with respect the norm $\|\cdot\|_{L^p(Y,\rho)}$.
\end{lem}

\begin{proof}
  Let $f\in L^p(Y,\rho)$. For $k\in\N$ we put
  \[f_k(x)=\left\{\begin{split}
    &k &&\text{if }f(x)> k\\
    &f(x) &&\text{if }f(x)\in[-k,k]\\
    &-k &&\text{if }f(x)<-k\\
  \end{split}\right.\]
  so that $f_k(x)$ is bounded and measurable. Then by \textbf{Lemma \ref{lemma:appr_BUC}} there exists $\widetilde{f}_k\in BUC(Y)$ such that
  \[\rho(\{x\in Y:\ \widetilde{f}_k(x)\neq f_k(x)\})\leq\frac{1}{2^k}\]
  Then by \cite{Miculescu} there exists $g_k\in Lip_b(Y)$ such that
  \[\|g_k-\widetilde{f}_k\|_{L^\infty(Y)}\leq\frac{1}{2^k}.\]
  Now we estimate
  \[\|g_k-f\|_{L^p(Y,\rho)}\leq \|g_k-\widetilde{f}_k\|_{L^p(Y,\rho)}+\|\widetilde{f}_k-f_k\|_{L^p(Y,\rho)}+\|f_k-f\|_{L^p(Y,\rho)},\]
  where
  \[\begin{split}
     \|g_k-\widetilde{f}_k\|_{L^p(Y,\rho)}&=\left(\int_Y|g_k(x)-\widetilde{f}_k(x)|^p\rho(dx)\right)^{1/p}\\ &\leq\|g_k-\widetilde{f}_k\|_{L^\infty(Y)}\rho(Y)^{1/p}\leq\frac{\rho(Y)^{1/p}}{2^k}.
  \end{split}\]
  Concerning the second term we recall that
  \[\sup_{x\in Y}|\widetilde{f}_k(x)|\leq2\sup_{x\in Y}|f_k(x)|=2k\]
  then
  \[\begin{split}
     \|\widetilde{f}_k-f_k\|_{L^p(Y,\rho)}&=\left(\int_Y|\widetilde{f}_k(x)-f_k(x)|^p\rho(dx)\right)^{1/p}\\
     &=\left(\int_{\{x\in Y:\ \widetilde{f}_k(x)\neq f_k(x)\}}|\widetilde{f}_k(x)-f_k(x)|^p\rho(dx)\right)^{1/p}\\
     &\leq 3k\ \rho(\{x\in Y:\ \widetilde{f}_k(x)\neq f_k(x)\})^{1/p}\leq\frac{3k}{2^{k/p}}.
  \end{split}\]
  Finally we remark that since $f_k\rightarrow f$ $\rho$-a.e. for $k\rightarrow\infty$, and $|f_k(x)|\leq|f(x)|\in L^p(Y,\rho)$, then the Lebesgue theorem yields
  \[\|f_k-f\|_{L^p(Y,\rho)}\rightarrow0,\ k\rightarrow\infty.\]
\end{proof}

\end{document}